\documentclass[a4paper,11pt]{article}

\usepackage{geometry} 
\usepackage{amsfonts}
\usepackage{amsmath,amsthm}
\usepackage{graphicx}
\usepackage{amssymb}
\usepackage{tikz}
\usepackage{tikz-cd}
\usepackage{color}

\newtheorem{theorem}{Theorem}[section]

\newtheorem{corollary}[theorem]{Corollary}

\newtheorem{definition}[theorem]{Definition}

\newtheorem{lemma}[theorem]{Lemma}

\newtheorem{proposition}[theorem]{Proposition}

\begin{document}

\title{Principal bundle structure of matrix manifolds}
\author{M. Billaud-Friess\thanks{Department of Computer Science and Mathematics, 
GeM, Ecole Centrale de Nantes, 
1 rue de la No\"e, BP 92101, 44321 Nantes Cedex 3, France. Email: [marie.billaud-friess,anthony.nouy]@ec-nantes.fr.} , 
A. Falc\'o\footnotemark[1] \thanks{Departamento de Matem\'aticas, F\'{\i}sica y Ciencias Tecnol\'ogicas, Universidad CEU Cardenal Herrera, San Bartolom\'e 55, 46115 Alfara del Patriarca (Valencia), Spain. E-mail: {afalco@uchceu.es}.} ,   A. Nouy\footnotemark[1]}
\date{}
\maketitle

\begin{abstract}
In this paper, we introduce a new geometric description of the manifolds of matrices of fixed rank.
The starting point is a geometric description of the Grassmann manifold $\mathbb{G}_r(\mathbb{R}^k)$ of linear subspaces of dimension $r<k$ in $\mathbb{R}^k$ which avoids the use of equivalence classes. The set $\mathbb{G}_r(\mathbb{R}^k)$ is equipped with an atlas which provides it with the structure of an analytic manifold modelled on $\mathbb{R}^{(k-r)\times r}$.  
Then we define an atlas for the set  $\mathcal{M}_r(\mathbb{R}^{k \times r})$ of full rank matrices and prove that the resulting manifold is an analytic principal bundle with base  $\mathbb{G}_r(\mathbb{R}^k)$ and typical fibre $\mathrm{GL}_r$, the general linear group of invertible matrices in $\mathbb{R}^{k\times k}$. Finally, we define an atlas for the set $\mathcal{M}_r(\mathbb{R}^{n \times m})$ of non-full rank matrices and prove that the resulting 
manifold is an analytic principal bundle with base  $\mathbb{G}_r(\mathbb{R}^n) \times \mathbb{G}_r(\mathbb{R}^m)$ and typical fibre $\mathrm{GL}_r$. 
The atlas of $\mathcal{M}_r(\mathbb{R}^{n \times m})$ is indexed on the manifold itself, which allows a natural definition of a neighbourhood for a given matrix, this neighbourhood being proved to possess the structure of a Lie group. 
Moreover, the set $\mathcal{M}_r(\mathbb{R}^{n \times m})$ equipped with the topology induced by the atlas is 
proven to be an embedded submanifold of the matrix space $\mathbb{R}^{n \times m}$ 
equipped with the subspace topology. The proposed geometric description then results in a description of the matrix space $\mathbb{R}^{n \times m}$, seen as the union of manifolds $\mathcal{M}_r(\mathbb{R}^{n \times m})$, as an analytic manifold equipped with a topology for which the matrix rank is a continuous map. 
\end{abstract}




\section{Introduction}

Low-rank matrices appear in many applications involving high-dimensional data. Low-rank models are commonly used in statistics, machine learning  or data analysis (see \cite{Zhou2014} for a recent survey).  Also, low-rank approximation of matrices is the cornerstone of many modern numerical methods for high-dimensional problems in computational science, such as model order reduction methods for dynamical systems, or parameter-dependent or stochastic equations \cite{antoulas2001survey,benner2015survey,Nouy2016,morbook2017}.
\par
These applications yield problems of approximation or optimization in the sets of matrices with fixed rank
$$
\mathcal{M}_r(\mathbb{R}^{n\times m})=\{
Z\in \mathbb{R}^{n\times m}: \mathrm{rank}(Z) = r
\}.
$$
A usual geometric approach is to endow the set 
 $\mathcal{M}_r(\mathbb{R}^{n\times m})$ with the structure of a Riemannian manifold \cite{smith1994optimization,AMSb}, which is seen as an embedded submanifold of 
$\mathbb{R}^{n\times m}$ equipped with the topology $\tau_{\mathbb{R}^{n\times m}}$ given by matrix norms. Standard algorithms then work in the ambient matrix space $\mathbb{R}^{n\times m}$ and do not rely on an explicit geometric description of the manifold using local charts (see, e.g., \cite{vandereycken2013low,mishra2013low,mishra2014fixed,koch2007}). However, the matrix rank considered as a map is not continuous for the topology  $\tau_{\mathbb{R}^{n\times m}}$,  which can yield undesirable numerical issues.

The purpose of this paper is to propose a new geometric description of 
 the sets of matrices with fixed rank which is amenable for numerical use, and which relies on the natural 
 parametrization of matrices in $\mathcal{M}_r(\mathbb{R}^{n\times m})$  given by
\begin{align}\label{matrix_representation}
Z = U G V^T,
\end{align}
where $U \in \mathbb{R}^{n \times r}$ and $V \in \mathbb{R}^{m\times r}$ are matrices with full rank $r<\min\{n,m\}$, and
$G \in \mathbb{R}^{r \times r}$ is a non singular matrix. 
The set $\mathcal{M}_r(\mathbb{R}^{n\times m})$ is here endowed with the structure of analytic principal bundle, with an explicit description of local charts.   This results  in a description of the matrix  space $\mathbb{R}^{n\times m}$ as an analytic manifold with a topology induced by local charts which is different from $\tau_{\mathbb{R}^{n\times m}}$ and for which the rank is a continuous map. Note that the representation  \eqref{matrix_representation} of a matrix $Z$ is not unique because
$
Z = (UP)(P^{-1}GP^T)(VP^{-1})^T
$
holds for every invertible matrix $P$ in  $\mathbb{R}^{r \times r}$. An argument used to dodge this undesirable property is the possibility to uniquely define a tangent space (see for example Section 2.1 in \cite{koch2007}), which is a prerequisite for standard algorithms on differentiable manifolds. The geometric description proposed in this paper avoids this undesirable property. Indeed, the system of local charts for the set $\mathcal{M}_r(\mathbb{R}^{n\times m})$ is indexed on the set itself. This allows a natural definition of a neighbourhood for a matrix where all matrices   
admit a unique representation. 

The present work opens the route for new numerical methods for optimization or dynamical low-rank approximation, with algorithms working in local coordinates and avoiding the use of a Riemannian structure, such as in 
\cite{Manton}, where a framework  is introduced
for generalising iterative methods from Euclidean space to manifolds which ensures that local convergence rates are preserved. 
The introduction of a principal bundle representation of matrix manifolds
is also motivated by the importance of this geometric structure in the concept of gauge potential in physics \cite{Michor}. 

We would point out that the 
 proposed geometric description has a natural extension to the case of fixed-rank operators on infinite dimensional spaces
and is consistent with the geometric description of manifolds of tensors with fixed rank 
proposed by Falc\'o, Hackbush and Nouy \cite{falco2016dirac}, in a tensor 
Banach space framework.

Before introducing the main results and outline of the paper, we recall some elements of geometry.

\subsection{Elements of geometry}\label{subsec:elements_geometry}

In this paper, we follow the approach of Serge Lang \cite{Lang} for the definition of a manifold $\mathbb{M}$.
In this framework, a set $\mathbb{M}$ is equipped with an atlas 
which gives $\mathbb{M}$ the structure of a topological space, with a topology induced by local charts, and the structure of differentiable manifold compatible with this topology. 
More precisely, the starting point is the definition of
 a collection of  non-empty subsets $U_\alpha \subset \mathbb{M}$, with $\alpha$ in a set $A$, such that $\{U_\alpha\}_{\alpha \in A}$ is a covering of $\mathbb{M}$. 
The next step is the explicit construction for any $\alpha \in A$ of a local chart $\varphi_\alpha$ which is 
  a bijection  from $U_\alpha$ to an open set $X_\alpha$ of the finite dimensional space $\mathbb{R}^{N_{\alpha}}$ such that 
for any pair $\alpha,\alpha'\in \mathbb{M}$ such that $U_\alpha \cap U_{\alpha'} \neq \emptyset$, the following properties hold: 
 \begin{enumerate}
\item[(i)] $\varphi_\alpha(U_\alpha \cap U_{\alpha'})$ and $\varphi_{\alpha'}(U_\alpha \cap U_{\alpha'})$ are open sets in $X_{\alpha}$ and $X_{\alpha'}$ respectively, and
\item[(ii)] the map 
$$
\varphi_{\alpha'}  \circ \varphi_\alpha ^{-1}: \varphi_\alpha(U_\alpha \cap U_{\alpha'}) \longrightarrow \varphi_{\alpha'} (U_\alpha \cap U_{\alpha'})
$$
is a $\mathcal{C}^p$ differentiable diffeomorphism, with $p \in \mathbb{N}\cup \{\infty\}$ or $p=\omega$ when the map is analytic.
\end{enumerate}
Under the above assumptions, the set $\mathcal{A}:=\{(U_\alpha,\varphi_\alpha): \alpha \in A\}$
is an atlas which endows $\mathbb{M}$ with a structure of $\mathcal{C}^{p}$ manifold. Then
we say that $(\mathbb{M},\mathcal{A})$ is a $\mathcal{C}^{p}$ manifold, or an analytic manifold when $p=\omega$.  A  consequence of the condition $(ii)$ is that when $U_\alpha \cap U_{\alpha'} \neq \emptyset$ holds for $\alpha,\alpha'\in A$, then
$N_\alpha = N_{\alpha'}.$ 
In the particular case where $N_\alpha=N$ for all $\alpha\in A$,  
we say that $(\mathbb{M},\mathcal{A})$ is a $\mathcal{C}^p$ manifold 
modelled on $\mathbb{R}^N.$ Otherwise, we say that it is a manifold
not modelled on a particular finite-dimensional space.  A paradigmatic example is the Grassmann manifold $\mathbb{G}(\mathbb{R}^k)$ of all 
linear subspaces of $\mathbb{R}^k$, such that
$$
\mathbb{G}(\mathbb{R}^k)=\bigcup_{0 \le r \le k}\mathbb{G}_r(\mathbb{R}^k),
$$
where $\mathbb{G}_0(\mathbb{R}^k)= \{0\}$ and  $\mathbb{G}_k(\mathbb{R}^k) =\{\mathbb{R}^k\}$ are trivial manifolds 
and $\mathbb{G}_r(\mathbb{R}^k)$ is a manifold modelled on the linear space $\mathbb{R}^{(k-r)\times r}$
for $0 < r < k.$ In consequence, $\mathbb{G}(\mathbb{R}^k)$ is a manifold not modelled on a
particular finite-dimensional space. 

The atlas also endows $\mathbb{M}$ with a topology given by 
$$
\tau_{\mathcal{A}}:=\left\{
\varphi_\alpha^{-1}(O): \alpha \in A  \text{ and } O \text{ an open set in } X_\alpha
\right\},
$$
which makes $(\mathbb{M},\tau_{\mathcal{A}})$  a topological space where 
each local chart 
$$
\varphi_{\alpha}:(U_\alpha,\tau_{\mathcal{A}}|_{U_\alpha})
\longrightarrow (X_\alpha,\tau_{\mathbb{R}^{N_\alpha}}|_{X_{\alpha}}),
$$
considered as a map between topological spaces, 
is a homeomorphism.\footnote{Here $(\mathfrak{X},\tau)$ denotes a topological space and if $\mathfrak{X}' \subset \mathfrak{X}$, then
$\tau|_{\mathfrak{X}'}$ denotes the subspace topology.
}

\subsection{Main results and outline}

Our first remark is that the matrix space $\mathbb{R}^{n \times m}$ is an analytic manifold modelled on itself and its geometric structure is fully compatible with  the topology $\tau_{\mathbb{R}^{n \times m}}$ induced by a matrix norm. In this paper, we define an atlas on 
$\mathcal{M}_r(\mathbb{R}^{n\times m})$ which gives this set the structure of an 
analytic manifold, with a topology induced by the atlas fully
compatible with the subspace topology 
$\tau_{\mathbb{R}^{n \times m}}|_{\mathcal{M}_r(\mathbb{R}^{n\times m})}$. This implies that
$\mathcal{M}_r(\mathbb{R}^{n\times m})$ is an embedded submanifold of the matrix manifold $\mathbb{R}^{n \times m}$ modelled on itself\footnote{Note that the set $\mathcal{M}_0(\mathbb{R}^{n\times m})=\{0\}$ 
is a trivial manifold, which is trivially embedded in  $\mathbb{R}^{n\times m}.$}. 
For the topology $\tau_{\mathbb{R}^{n \times m}}$, the matrix rank considered as
a map is not continuous but only lower semi-continuous. 
However, if $\mathbb{R}^{n \times m}$ is seen as the disjoint union of sets of matrices with fixed rank, 
 \begin{align}\label{two_faces}
\mathbb{R}^{n\times m} = \bigcup_{0 \le r \le \min\{n,m\}} \mathcal{M}_r(\mathbb{R}^{n\times m}),
\end{align}
then $\mathbb{R}^{n\times m}$ has the structure of an analytic manifold not modelled 
on a particular finite-dimensional space equipped with a topology 
$$
\tau_{\mathbb{R}^{n \times m}}^* = 
\bigcup_{0 \le r \le \min\{n,m\}}\tau_{\mathbb{R}^{n \times m}}|_{\mathcal{M}_r(\mathbb{R}^{n\times m})},
$$
which is not equivalent to $\tau_{\mathbb{R}^{n \times m}}$, and 
for which the matrix rank is a continuous map.
\\\par
Note that in the case when $r=n=m$, the set $\mathcal{M}_n(\mathbb{R}^{n\times n})$ coincides with the general linear group $\mathrm{GL}_n$ of invertible matrices in 
$\mathbb{R}^{n\times n},$ which is an analytic manifold trivially embedded in $\mathbb{R}^{n\times n}$. In all other cases, which are addressed in this paper,
our geometric description of $\mathcal{M}_r(\mathbb{R}^{n\times m})$ relies on a geometric description of the Grassmann manifold $\mathbb{G}_{r}(\mathbb{R}^k)$, with $k=n$ or $m$.
\par
Therefore, we start in Section  \ref{sec:grassmann} by introducing a geometric description of $\mathbb{G}_{r}(\mathbb{R}^k)$.
 A classical approach consists of describing $\mathbb{G}_{r}(\mathbb{R}^k)$  as the quotient manifold $\mathcal{M}_{r}(\mathbb{R}^{k \times r})/\mathrm{GL}_r$ of equivalent classes of full-rank matrices $Z$ in $\mathcal{M}_r(\mathbb{R}^{k\times r})$ having the same column space $\mathrm{col}_{k,r}(Z)$.  Here, we avoid the use of equivalent classes and provide 
an explicit description of an atlas $\mathcal{A}_{k,r}=
\{(\mathfrak{U}_Z,\varphi_{Z})\}_{Z \in \mathcal{M}_{r}(\mathbb{R}^{k \times r})}$  for $\mathbb{G}_{r}(\mathbb{R}^k)$, with local chart
$$
\varphi_Z : \mathfrak{U}_Z \to \mathbb{R}^{(k-r)\times r}  , \quad \varphi_Z^{-1}(X) = \mathrm{col}_{k,r}(Z+Z_\bot X),
$$ 
where $Z_\bot \in \mathbb{R}^{k\times (k-r)}$ is such that $Z_\bot^TZ=0$ and 
$\mathrm{col}_{k,r}(A)$ denotes the column space of a matrix $A \in \mathbb{R}^{k\times r}$, and we prove that the neighbourhood  $ \mathfrak{U}_Z$   have the structure of a Lie group.  This parametrization of the Grassmann manifold is introduced in \cite[Section 2]{AMS} but the authors do not elaborate on it. 

Then in Section \ref{sec:fullrank}, we consider the particular case of full-rank matrices. We introduce an atlas $\mathcal{B}_{k,r} = \{(\mathcal{V}_Z,\xi_{Z})\}_{Z \in \mathcal{M}_r(\mathbb{R}^{k\times r})}$ 
for the  manifold $\mathcal{M}_{r}(\mathbb{R}^{k \times r})$ of matrices with full rank $r<k$, with local chart
$$
\xi_Z : \mathcal{V}_Z\to \mathbb{R}^{ (k-r)\times r} \times \mathrm{GL}_r , \quad 
\xi_Z^{-1}(X,G) = (Z+Z_\bot X)G, 
$$
and  prove that  $\mathcal{M}_{r}(\mathbb{R}^{k \times r})$ is an analytic principal bundle with base $\mathbb{G}_{r}(\mathbb{R}^{k})$ and typical fibre $\mathrm{GL}_r$. Moreover, we prove that $\mathcal{M}_{r}(\mathbb{R}^{k \times r})$ is an embedded submanifold of $(\mathbb{R}^{k \times r},\tau^*_{\mathbb{R}^{k\times r}})$, and that each of the neighbourhoods $\mathcal{V}_Z$  have the structure of a Lie group.

Finally, in Section \ref{sec:fixed-rank}, we provide an analytic atlas $\mathcal{B}_{n,m,r} = \{(\mathcal{U}_Z,\theta_{Z})\}_{Z \in \mathcal{M}_r(\mathbb{R}^{n\times m})}$  for the set  $\mathcal{M}_{r}(\mathbb{R}^{n\times m})$
of matrices $Z=UGV^T$ with rank $r< \min\{n,m\}$, with local chart 
$$
\theta_{Z} :\mathcal{U}_Z\to \mathbb{R}^{(n-r) \times r} \times \mathbb{R}^{(m-r) \times r}  \times \mathrm{GL}_r , \quad 
\theta_{Z}^{-1}(X,Y,H) = (U+U_\bot X) H(V+V_\bot Y), 
$$
and we prove that $\mathcal{M}_{r}(\mathbb{R}^{n\times m})$ is {an analytic} principal bundle with base $\mathbb{G}_{r}(\mathbb{R}^{n}) \times \mathbb{G}_{r}(\mathbb{R}^{m})$ and typical fibre $\mathrm{GL}_r$. Then we prove that $\mathcal{M}_r(\mathbb{R}^{n\times m})$ is an embedded submanifold of $(\mathbb{R}^{n \times m},\tau^*_{\mathbb{R}^{n\times m}})$, and that each of the neighbourhoods $\mathcal{U}_Z$  have the structure of a Lie group.


\section{The Grassmann manifold $\mathbb{G}_r(\mathbb{R}^k)$}\label{sec:grassmann}

In this section, we present a geometric description of the Grassmann manifold $\mathbb{G}_r(\mathbb{R}^k)$ of all subspaces of dimension $r$ in $\mathbb{R}^k$, $0< r<k$, 
$$
\mathbb{G}_r(\mathbb{R}^k)=\{\mathcal{V} \subset \mathbb{R}^k: \mathcal{V} \text{ is a linear subspace  with } \dim (\mathcal{V}) = r\},
$$
with an explicit description of local charts. 
We first introduce the surjective map 
$$
\mathrm{col}_{k,r}:\mathcal{M}_{r}(\mathbb{R}^{k\times r}) \longrightarrow \mathbb{G}_r(\mathbb{R}^k),
\quad Z \mapsto \mathrm{col}_{k,r}\,(Z),
$$
where $\mathrm{col}_{k,r}(Z)$ is the {column space} of the matrix $Z$, which is the subspace spanned by the column vectors of $Z.$  
Given $\mathcal{V} \in \mathbb{G}_r(\mathbb{R}^k),$ there are infinitely many matrices $Z$ such that 
$ \mathrm{col}_{k,r}(Z) = \mathcal{V} $.  
Given a matrix $Z \in \mathcal{M}_r(\mathbb{R}^{k\times r})$, the set of matrices in $\mathcal{M}_r(\mathbb{R}^{k\times r})$ having the same column space as $Z$ is
$$
Z\mathrm{GL}_r:=\{ZG: G \in \mathrm{GL}_r\}.
$$

\subsection{An atlas for $\mathbb{G}_r(\mathbb{R}^k)$}

For a given matrix $Z$ in $\mathcal{M}_r(\mathbb{R}^{k\times r})$, we let $Z_{\bot} \in \mathcal{M}_{k-r}(\mathbb{R}^{k\times (k-r)})$ be a matrix
 such that $Z^TZ_{\bot}=0$ and we introduce an \emph{affine cross section}
\begin{align}
\mathcal{S}_{Z}:=\{W \in \mathcal{M}_r(\mathbb{R}^{k\times r}): Z^TW = Z^TZ\}, \label{SZ_affine}
\end{align}
which has the following equivalent characterization. 

\begin{lemma}\label{lemma:SZ}
The affine cross section $\mathcal{S}_{Z}$ is characterized by
\begin{align}
\mathcal{S}_{Z}=\{Z+Z_{\bot}X: X \in  \mathbb{R}^{(k-r) \times r}\}, \label{SZ}\end{align}
and the map
$$\eta_{Z}:\mathbb{R}^{(k-r) \times r} \longrightarrow \mathcal{S}_{Z}, \quad X \mapsto Z+Z_{\bot}X$$
is bijective.
\end{lemma}

\begin{proof}
We first observe that  
$
Z^T(Z+Z_{\bot}X_Z)= Z^TZ
$ for all $X \in  \mathbb{R}^{(k-r) \times r}$, 
which implies that $\{Z+Z_{\bot}X: X \in  \mathbb{R}^{(k-r) \times r}\} \subset \mathcal{S}_{Z}.$ For the other inclusion, we observe that if $W \in \mathcal{S}_{Z}$, then $Z^TW=Z^TZ$ and hence $W-Z \in 
\mathrm{col}_{k,r}(Z)^{\bot} $, the orthogonal subspace to $\mathrm{col}_{k,r}(Z)$ in $\mathbb{R}^k$.
Since $\mathrm{col}_{k,r}(Z)^{\bot} =  \mathrm{col}_{k,k-r}\,(Z_{\bot}),$ there exists
$X \in \mathbb{R}^{(k-r)\times r}$ such that
$W-Z = Z_{\bot}X.$  Proving that $\eta_{Z}$ is bijective is straightforward.
\end{proof}

\begin{proposition}\label{geometryU}
For  each $W  \in \mathcal{M}_r(\mathbb{R}^{k\times r})$ such that $\det(Z^TW) \neq 0$, there exists a unique $G_W \in \mathrm{GL}_r$ such that
$$
W\mathrm{GL}_r \cap \mathcal{S}_{Z} = \{WG_W^{-1}\}
$$ 
holds, which means that  the set of matrices with the same column space as $W$ intersects $\mathcal{S}_{Z}$ 
at the single point $WG_W^{-1}.$ 
Furthermore, $G_W=id_{r}$ 
if and only if $W \in \mathcal{S}_{Z}.$
\end{proposition}

\begin{proof}
By Lemma \ref{lemma:SZ}, a matrix $A \in W\mathrm{GL}_r \cap \mathcal{S}_{Z} $ is such that 
$A = WG_W^{-1} = Z+Z_{\bot}X $ for a certain $G_W \in \mathrm{GL}_r$ and a certain $X \in \mathbb{R}^{(k-r)\times r}$. Then $Z^TWG_W^{-1}=Z^TZ$ and $G_W$ is uniquely defined by $G_W=(Z^TZ)^{-1}(Z^TW),$ which proves that $W\mathrm{GL}_r \cap \mathcal{S}_{Z}$ is the singleton $\{WG_W^{-1}\},$ and 
$G_W =  id_{r}$ 
if and only if $W \in \mathcal{S}_{Z}.$
\end{proof}

\begin{corollary}\label{colbijective}
For each $Z \in \mathcal{M}_r(\mathbb{R}^{k\times r})$, 
 the map $\mathrm{col}_{k,r}:\mathcal{S}_{Z} \longrightarrow \mathbb{G}_r(\mathbb{R}^k)$
is injective.
\end{corollary}

\begin{proof} 
Let us assume the existence of $W,\tilde W\in \mathcal{S}_{Z}$ 
such that $\mathrm{col}_{k,r}(W)=\mathrm{col}_{k,r}(\tilde W).$ Then $W=\tilde W$ by Proposition \ref{geometryU}.
\end{proof}

Lemma \ref{lemma:SZ} and Corollary \ref{colbijective} allow us 
to construct a system of local charts for $\mathbb{G}_r(\mathbb{R}^k)$ by 
defining for each $Z \in \mathcal{M}_r(\mathbb{R}^{k\times r})$ a neighbourhood of $\mathrm{col}_{k,r}(Z)$ by
$$
\mathfrak{U}_{Z}:= \mathrm{col}_{k,r}(\mathcal{S}_{Z}) = 
\{ \mathrm{col}_{k,r}\,(W): W\in \mathcal{S}_{Z}\}
$$
together with the bijective map
$$\varphi_Z:= (\mathrm{col}_{k,r} \circ \eta_{Z})^{-1} : \mathfrak{U}_{Z} \to \mathbb{R}^{(k-r) \times r}$$
such that 
$$
\varphi_{Z}^{-1}(X) = \mathrm{col}_{k,r}(Z+Z_{\bot}X)
$$
for $X\in  \mathbb{R}^{(k-r)\times r}$.
We denote by $Z^+$ the Moore-Penrose pseudo-inverse of the full rank matrix $Z\in \mathcal{M}_{r}(\mathbb{R}^{r\times k})$, defined by 
$$
Z^{+}:= (Z^TZ)^{-1}Z^T \in \mathcal{M}_{r}(\mathbb{R}^{r\times k}).
$$ 
It satisfies $Z^+Z = id_{r}$ and $Z^+Z_\bot=0$. Moreover, $ZZ^+ \in \mathbb{R}^{k\times k}$ is the projection onto $\mathrm{col}_{k,r}(Z)$ parallel to $\mathrm{col}_{k,r}(Z)^\bot$.
Finally, we have the following result.

\begin{theorem}\label{grassmann}
The collection $\mathcal{A}_{k,r}:=\{(\mathfrak{U}_{Z},\varphi_{Z}):Z \in \mathcal{M}_{r}(\mathbb{R}^{k\times r})\}$ is
an analytic atlas for $\mathbb{G}_r(\mathbb{R}^k)$ and hence $(\mathbb{G}_r(\mathbb{R}^k),\mathcal{A}_{k,r})$ 
is an analytic $r(k-r)$-dimensional manifold modelled on $\mathbb{R}^{(k-r) \times r}$.
\end{theorem}

\begin{proof}
Clearly $\{\mathfrak{U}_{Z}\}_{Z \in \mathcal{M}_{r}(\mathbb{R}^{k\times r})}$ is a covering of $\mathbb{G}_r(\mathbb{R}^k).$ Now let 
$Z$ and $\tilde Z$ be such that $\mathfrak{U}_{Z} \cap \mathfrak{U}_{\tilde Z} \neq \emptyset$. 
Let $\mathcal{V} \in \mathfrak{U}_{Z}$ such that $\mathcal{V} = \varphi_Z^{-1}(X)= \mathrm{col}_{k,r}(Z+Z_\bot X)$, with $X\in \mathbb{R}^{k\times (k-r)}$. We can write $Z+Z_\bot X  = (\tilde Z + \tilde Z_\bot \tilde X)G$ with $G = \tilde Z^+(Z+Z_\bot X)$ and $\tilde X =  \tilde Z_\bot^+ (Z+Z_\bot X )G^{-1}$. Therefore,  $\mathcal{V} =  \mathrm{col}_{k,r}((\tilde Z + \tilde Z_\bot \tilde X)G) = \mathrm{col}_{k,r}(\tilde Z + \tilde Z_\bot \tilde X) = \varphi_{\tilde Z}^{-1}({\tilde X}) \in \mathfrak{U}_{\tilde Z}$, which implies that 
  $\mathfrak{U}_{Z} = \mathfrak{U}_{Z} \cap  \mathfrak{U}_{\tilde Z}$. Therefore, $\varphi_Z(\mathfrak{U}_{Z} \cap  \mathfrak{U}_{\tilde Z}) = \varphi_Z(\mathfrak{U}_{Z}) = \mathbb{R}^{k\times (n-k)}$ is an open set. In the same way, we show that $\mathfrak{U}_{\tilde Z} = \mathfrak{U}_{Z} \cap  \mathfrak{U}_{\tilde Z}$  and $ \varphi_{\tilde Z}(\mathfrak{U}_{Z}) = \mathbb{R}^{k\times (n-k)}$ is an open set. Finally, the map $\varphi_{\tilde Z} \circ \varphi_Z^{-1}$ from $\mathbb{R}^{(k-r) \times r}$ to $\mathbb{R}^{(k-r) \times r}$  
is given by $\varphi_{\tilde Z} \circ \varphi_{Z}^{-1}(X) =  \tilde Z_{\bot}^{+}(Z+Z_{\bot} X)G^{-1}
$, with $G = \tilde Z^+ (Z + Z_\bot X_Z) $, which is clearly an analytic map. 
\end{proof}

\subsection{Lie group structure of neighbourhoods $\mathfrak{U}_Z$} \label{sec:liegroupSZ}

Here we prove that each neighbourhood $\mathfrak{U}_Z$ of $\mathbb{G}_r(\mathbb{R}^k)$ is a Lie group.  
For that, we first note that  
a neighbourhood $\mathfrak{U}_Z$ of $\mathbb{G}_r(\mathbb{R}^k)$ can be identified with the set $\mathcal{S}_{Z}$ through the application $\mathrm{col}_{k,r} : \mathcal{S}_{Z} \to \mathfrak{U}_Z$. The next step is to 
identify $\mathcal{S}_{Z}$ with 
a closed Lie subgroup of $\mathrm{GL}_k,$ denoted by $\mathcal{G}_{Z},$ with associated Lie algebra $\mathfrak{g}_Z$ isomorphic to $\mathbb{R}^{r \times (k-r)}$, and such that the exponential map\footnote{We recall that the matrix exponential $\exp:\mathbb{R}^{k\times k} \rightarrow \mathrm{GL}_k$
is defined by 
$
\exp(A)=\sum_{n=0}^{\infty} \frac{A^n}{n!}.
$
} $\exp: \mathfrak{g}_Z \longrightarrow \mathcal{G}_{Z}$ is a diffeomorphism. To this end, for a given
$Z \in \mathcal{M}_r(\mathbb{R}^{k \times r})$, we introduce the vector space
\begin{align}
\mathfrak{g}_{Z}:=
\{Z_{\bot}XZ^+: X \in  \mathbb{R}^{(k-r)\times r}\} \subset \mathbb{R}^{k \times k}. \label{hZ}
\end{align} 
The following proposition proves that $\mathfrak{g}_{Z}$ is a
commutative subalgebra of $\mathbb{R}^{k \times k}.$ 

\begin{proposition}\label{algebra}
For all $X, \tilde X\in  \mathbb{R}^{(k-r)\times r}$,   
$$(Z_{\bot}XZ^+)(Z_{\bot}\tilde X Z^+)=0$$ 
holds, and $\mathfrak{g}_{Z}$ is a
commutative subalgebra of $\mathbb{R}^{k \times k}.$ Moreover, 
\begin{align}\label{exp}
\exp(Z_{\bot}XZ^+)=id_{k} + Z_{\bot}XZ^+,
\end{align}
\begin{align}\label{exp1}
\exp(Z_{\bot}XZ^+)Z = Z+Z_{\bot}X,
\end{align}
and
\begin{align}\label{exp11}
\exp(Z_{\bot}XZ^+)Z_{\bot} = Z_{\bot}
\end{align}
hold for all $X \in \mathbb{R}^{(k-r)\times r}.$ 
\end{proposition}

\begin{proof}
Since
$
(Z_{\bot}X Z^+)(Z_{\bot}\tilde X Z^+)=0
$
holds for all $X,\tilde X \in  \mathbb{R}^{(k-r)\times r},$  the vector space
$\mathfrak{g}_{Z}$ is a closed subalgebra of 
the matrix unitary algebra $\mathbb{R}^{k \times k}.$ As a consequence, $(Z_{\bot}X Z^+)^p=0$ holds for all $X\in \mathbb{R}^{(k-r)\times r} $ and all $p\ge 2$, which proves 
 \eqref{exp}. 
 We directly deduce \eqref{exp1} using $ZZ^+ = id_r$, and \eqref{exp11} using $Z^+Z_\bot = 0$.
 \end{proof}

From Proposition \ref{algebra} and the definition of 
 $\mathcal{S}_{Z}$, we obtain
the following results.

\begin{corollary}\label{second_coordinate}
The affine cross section $\mathcal{S}_Z$ satisfies
\begin{align}
\mathcal{S}_{Z}=\{\exp(Z_{\bot}XZ^+)Z: X \in \mathbb{R}^{(k-r)\times r}\},\label{SZLiegroup}
\end{align}
and 
\begin{align}
[\exp(Z_{\bot}XZ^+)Z \,| Z_{\bot}] \in \mathrm{GL}_k \label{expconcat}
\end{align}
for all $X \in \mathbb{R}^{(k-r)\times r}$, where the brackets $\left[\cdot|\cdot\right]$ are used for matrix concatenation.
\end{corollary}

\begin{proof} From Proposition \ref{algebra} and  \eqref{SZ}, we obtain \eqref{SZLiegroup} and we can write
$$
[\exp(Z_{\bot}XZ^+)Z \,| Z_{\bot}] = [\exp(Z_{\bot}XZ^+)Z \,| \exp(Z_{\bot}XZ^+)Z_{\bot}] = \exp(Z_{\bot}XZ^+)[Z|Z_{\bot}].
$$
Since $\exp(Z_{\bot}XZ^+),[Z|Z_{\bot}] \in \mathrm{GL}_k$, \eqref{expconcat} follows. 
\end{proof}

Now we need to introduce the following definition and proposition (see \cite[p.80]{procesi2007}).

\begin{definition}
Let $(\mathbb{K},+,\cdot)$ be a ring and let $(\mathbb{K},+)$ be its additive group. A subset $\mathbb{I} \subset \mathbb{K}$
is called a two-sided ideal (or simply an ideal) of $\mathbb{K}$ if it is an additive subgroup of $\mathbb{K}$ such that
$\mathbb{I}\cdot\mathbb{K}:=\{r \cdot x : r \in \mathbb{I} \text{ and } x \in \mathbb{K}\} \subset \mathbb{I}$ and
$\mathbb{K}\cdot\mathbb{I}:=\{x \cdot r : r \in \mathbb{I} \text{ and } x \in \mathbb{K}\} \subset \mathbb{I}.$
\end{definition}

\begin{proposition}\label{prop:procesi}
If $\mathfrak{g} \subset \mathfrak{h}$ is a two-sided ideal of the Lie algebra $\mathfrak{h}$ of a group $\mathcal{H}$, then the subgroup $\mathcal{G} \subset \mathcal{H}$ generated by 
$\exp(\mathfrak{g})=\{\exp(G): G \in \mathfrak{g}\}$ is normal and closed, with Lie algebra $\mathfrak{h}.$
\end{proposition}

From the above proposition, we deduce the following result. 

\begin{lemma}
Let $Z \in \mathcal{M}_r(\mathbb{R}^{k\times r})$ and 
$Z_{\bot} \in \mathcal{M}_{k-r}(\mathbb{R}^{k\times (k-r)})$ be such that $Z^TZ_{\bot}=0.$ Then $\mathfrak{g}_{Z} \subset \mathbb{R}^{k\times k}$ 
is a two-sided ideal of the Lie algebra $\mathbb{R}^{k\times k}$ and hence
\begin{align}
\mathcal{G}_{Z}:=\{\exp(Z_{\bot}X Z^+):X \in \mathbb{R}^{(k-r)\times r}\} \label{GZ}
\end{align}
is a closed Lie group 
with Lie algebra $\mathfrak{g}_{Z}.$ Furthermore, the map
$
\exp:\mathfrak{g}_{Z} \longrightarrow \mathcal{G}_{Z}
$
is bijective.
\end{lemma}

\begin{proof}
Consider $Z_{\bot}X Z^+ \in \mathfrak{g}_{Z}$ and
$A \in \mathbb{R}^{k\times k}$. Noting that $Z^{+}Z = id_{r}$ and $(Z_{\bot})^{+}Z_{\bot}=id_{k-r}$, we have that 
$$
(Z_{\bot}X Z^+)A = Z_{\bot}(XZ^{+}AZ)Z^+ ,
$$
which proves that 
$\mathfrak{g}_{Z}\cdot\mathbb{R}^{k\times k} \subset \mathfrak{g}_{Z}.$ 
Similarly, we have that
$$
A(Z_{\bot}XZ^+) = Z_{\bot}((Z_{\bot})^+AZ_{\bot}X)Z^{+},
$$
which proves that $\mathbb{R}^{k\times k} \cdot \mathfrak{g}_{Z} \subset \mathfrak{g}_{Z}.$ This proves that $ \mathfrak{g}_{Z}$ is a 
two-sided ideal. The map  $\exp$ is clearly
surjective. To prove that it is injective, we  assume
$\exp(Z_{\bot}XZ^+)=\exp(Z_{\bot}\tilde XZ^+)$ for $X,\tilde X \in  \mathbb{R}^{(k-r)\times r} $. 
Then from  \eqref{exp}, we obtain $Z+Z_{\bot}X=Z+Z_{\bot}\tilde X$
and hence $X=\tilde X$,  i.e. $Z_{\bot}XZ^+=Z_{\bot}\tilde XZ^+$ in $\mathfrak{g}_{Z}.$
\end{proof}

Finally, we can prove the following result.
\begin{theorem}
The set $\mathcal{S}_{Z}$ together with the group operation $\times_{Z}$ defined by
\begin{align}\label{Lie1}
\exp(Z_{\bot}XZ^+)Z \times_{Z} \exp(Z_{\bot}\tilde XZ^+)Z = \exp(Z_{\bot}(X+\tilde X)Z^+)Z 
\end{align}
for  $X,\tilde X \in \mathbb{R}^{(k-r)\times r}$ 
is a Lie group.
\end{theorem}

\begin{proof}
To prove that it is a Lie group, we simply note that the multiplication and inversion maps
$$
\mu:\mathcal{S}_{Z} \times \mathcal{S}_{Z} \longrightarrow \mathcal{S}_{Z},~
(W,\tilde W) \mapsto \exp(Z_{\bot}(Z_\bot^+ (W-Z) +Z_\bot^+(\tilde W-Z))Z^+)Z
$$
and
$$
\delta:\mathcal{S}_{Z}  \longrightarrow \mathcal{S}_{Z},~
W \mapsto \exp(-Z_{\bot} Z_\bot^+ (W-Z) Z^+)Z
$$
are analytic. 
\end{proof}
It follows that $\mathfrak{U}_Z$ can be identified with a Lie group through the map $\varphi_Z$.
\begin{theorem}
Each neighbourhood $\mathfrak{U}_Z$ of $\mathbb{G}_r(\mathbb{R}^k)$ together with the group operation $\circ_Z$ defined by 
$$
\mathcal{V} \circ_Z \mathcal{V'} = \varphi_{Z}^{-1}(\varphi_{Z}(\mathcal{V})+\varphi_{Z}(\mathcal{V}') )
$$
for  $\mathcal{V},\mathcal{V'}  \in \mathfrak{U}_Z$, is a Lie group and the map
$
\gamma_{Z}: \mathfrak{U}_Z \longrightarrow \mathcal{G}_{Z}  
$
given by
$$
\gamma_{Z}(\mathcal{U})= \exp(Z_{\bot}\varphi_Z(\mathcal{U})Z^+)
$$
is a Lie group isomorphism.
\end{theorem}


\section{The non-compact Stiefel principal bundle $\mathcal{M}_{r}(\mathbb{R}^{k\times r})$}\label{sec:fullrank}

In this section, we give a new geometric description of the set $\mathcal{M}_{r}(\mathbb{R}^{k\times r})$ of matrices with full rank $r< k$, which is based on the geometric description of the Grassmann manifold given in Section \ref{sec:grassmann}. 

\subsection{Principal bundle structure of $\mathcal{M}_{r}(\mathbb{R}^{k\times r})$}

For $Z \in \mathcal{M}_{r}(\mathbb{R}^{k\times r}) $, we define a neighbourhood of $Z$  as
 \begin{align}
\mathcal{V}_Z := \{W \in \mathcal{M}_r(\mathbb{R}^{k\times r}): \det(Z^TW) \neq 0\} \supset \mathcal{S}_{Z}. 
\end{align}
From Proposition~\ref{geometryU}, we know that for a given matrix $W \in \mathcal{V}_Z$,  
there exists a unique pair of matrices $(X,G) \in \mathbb{R}^{(k-r)\times r} \times \mathrm{GL}_r$ such that 
$W = (Z+Z_{\bot}X)G.$ Therefore, 
$$
\mathcal{V}_Z  = \{(Z+Z_{\bot}X)G : X \in \mathbb{R}^{(k-r)\times r}, G\in  \mathrm{GL}_r \}.
$$
 It allows us to introduce a parametrisation $\xi_Z^{-1}$ (see Figure \ref{param_Mkxr}) defined through the bijection  
\begin{align}
\xi_{Z}:\mathcal{V}_Z \longrightarrow \mathbb{R}^{(k-r) \times r} \times \mathrm{GL}_r, \label{xiZ}
\end{align}
such that 
$$
\xi_Z^{-1}(X,G) = (Z+Z_{\bot}X)G
$$
for $(X,G) \in \mathbb{R}^{(k-r) \times r} \times \mathrm{GL}_r$, and 
\begin{align*}
\xi_Z(W) = (Z_\bot^+W(Z^+W)^{-1},Z^+W)
\end{align*}
for $W\in \mathcal{V}_Z$.
 In particular, 
$$
\xi_Z^{-1}(0,id_r) = Z. 
$$
\begin{figure}[h]
\centering
\includegraphics[scale=1]{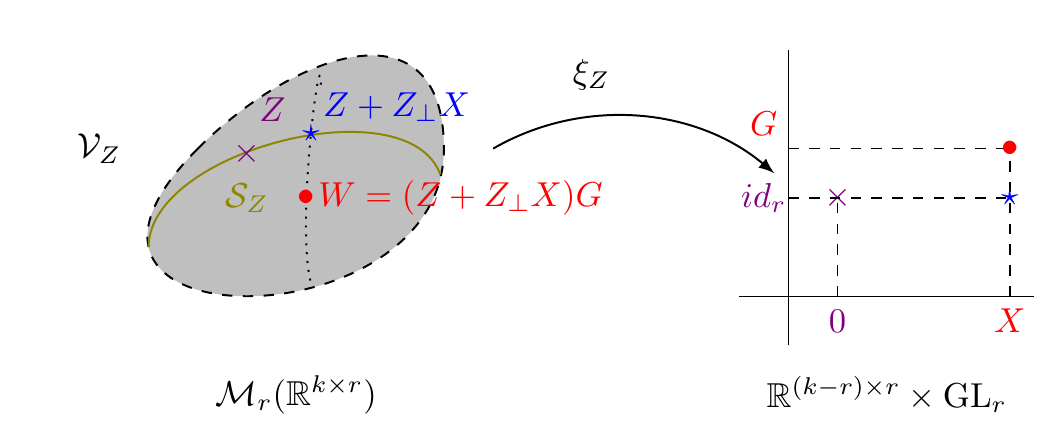}
\caption{Illustration of the chart $\xi_Z$ which associates to $W = (Z+Z_\bot X)G\in \mathcal{V}_Z \subset \mathcal{M}_{r}(\mathbb{R}^{k\times r})$ the parameters $(X,G)$ in $\mathbb{R}^{(k-r) \times r} \times \mathrm{GL}_r$.}
\label{param_Mkxr}
\end{figure}

\begin{theorem}\label{theoremB}
The collection $\mathcal{B}_{k,r}:=\{(\mathcal{V}_Z,\xi_{Z}):Z \in \mathcal{M}_r(\mathbb{R}^{k\times r})\}$ is
an analytic atlas for $\mathcal{M}_r(\mathbb{R}^{k\times r})$, 
and hence $(\mathcal{M}_r(\mathbb{R}^{k\times r}),\mathcal{B}_{k,r})$ it is an analytic $kr$-dimensional
manifold modelled on $\mathbb{R}^{(k-r) \times r} \times \mathbb{R}^{r \times r}.$
\end{theorem}

\begin{proof}
 $\{\mathcal{V}_Z\}_{Z \in \mathcal{M}_r(\mathbb{R}^{k\times r})}$ is clearly a covering
of $\mathcal{M}_r(\mathbb{R}^{k\times r})$. Moreover, 
since $\xi_{Z}$ is bijective from $\mathcal{V}_Z$ to $\mathbb{R}^{(k-r) \times r} \times \mathrm{GL}_r$ we claim that if $\mathcal{V}_Z \cap \mathcal{V}_{\tilde Z} \neq \emptyset$ for $Z,\tilde Z \in \mathcal{M}_r(\mathbb{R}^{k\times r}),$ then  the following statements hold:
\begin{enumerate}
\item[i)] $\xi_{Z}(\mathcal{V}_Z \cap \mathcal{V}_{\tilde Z})$  and  $\xi_{\tilde Z}(\mathcal{V}_Z \cap \mathcal{V}_{\tilde Z})$ are open sets 
in $\mathbb{R}^{(k-r)\times r} \times \mathrm{GL}_r$ and
\item[ii)]the map $
\xi_{\tilde Z} \circ \xi_{Z}^{-1}
$ is analytic from $\xi_{Z}(\mathcal{V}_Z \cap \mathcal{V}_{\tilde Z}) \subset
\mathbb{R}^{(k-r)\times r} \times \mathrm{GL}_r$ to $\xi_{\tilde Z}(\mathcal{V}_Z \cap \mathcal{V}_{\tilde Z}) \subset \mathbb{R}^{(k-r)\times r} \times \mathrm{GL}_r$.
\end{enumerate}
In this proof, we equip $\mathbb{R}^{k\times r}$ with the topology $\tau_{\mathbb{R}^{k\times r}}$ induced by matrix norms. For any $Z \in \mathcal{M}_r(\mathbb{R}^{k\times r})$, $\mathcal{V}_Z= \{W \in \mathbb{R}^{k\times r} : \det(Z^TW)\neq 0\}$ is the inverse image of the open set $\mathbb{R}\setminus \{0\}$ by the continuous map $W\mapsto \det(Z^TW)$ from $\mathbb{R}^{k\times r}$ to $\mathbb{R}$, and therefore, $\mathcal{V}_Z$ is an open set of $\mathbb{R}^{k\times r}$. 
Since $\mathcal{V}_Z$ and $\mathcal{V}_{\tilde Z}$ are open sets in $\mathbb{R}^{k\times r}$, $\mathcal{V}_Z \cap \mathcal{V}_{\tilde Z}$ is also an open set in $\mathbb{R}^{k\times r}$ and since $\xi_Z^{-1}$ is a continuous map from $ \mathbb{R}^{(k-r) \times r} \times \mathrm{GL}_r$ to $\mathbb{R}^{k\times r}$, the set $\xi_Z(\mathcal{V}_Z \cap \mathcal{V}_{\tilde Z})$, as the inverse image of an open set by a continuous map, is an open set in $\mathbb{R}^{(k-r) \times r} \times \mathrm{GL}_r$. Similarly, $\xi_{\tilde Z}(\mathcal{V}_Z \cap \mathcal{V}_{\tilde Z})$ is an open set.
Now let $(X,G) \in \mathbb{R}^{(k-r)\times r}\times \mathrm{GL}_r$ such that  
$ \xi_{Z}^{-1}(X,G)   \in \mathcal{V}_Z \cap \mathcal{V}_{\tilde Z}$. From the expressions of $\xi_Z^{-1}$ and $ \xi_{\tilde Z}$, the map $\xi_{\tilde Z}\circ \xi_{Z}^{-1}$ is defined by
\begin{align*}
\xi_{\tilde Z}\circ \xi_{Z}^{-1}(X,G)&= (\tilde Z_\bot^+ \xi_{Z}^{-1}(X,G) (\tilde Z^+ \xi_{Z}^{-1}(X,G) )^{-1},\tilde Z^+ \xi_{Z}^{-1}(X,G) ), 
\end{align*}
with $\xi_{Z}^{-1}(X,G) = (Z+Z_{\bot}X)G$, 
which is clearly an analytic map.
\end{proof}


Before stating the next result, we recall the definition of a morphism between manifolds and of 
a  fibre bundle. We introduce notions of $\mathcal{C}^p $ maps and $\mathcal{C}^p$ manifolds, 
 with $p\in \mathbb{N}\cup \{\infty\}$ or $p=\omega$. In the latter case, $\mathcal{C^\omega}$ means analytic. 
 
\begin{definition}
Let $(\mathbb{M},\mathcal{A})$ and $(\mathbb{N},\mathcal{B})$ 
be two $\mathcal{C}^p$  manifolds. Let $F:\mathbb{M} \rightarrow \mathbb{N}$ be a map.
We  say that $F$ is a $\mathcal{C}^p$ \emph{morphism} between $(\mathbb{M},\mathcal{A})$ and $(\mathbb{N},\mathcal{B})$ if given $m \in \mathbb{M}$, there exists a chart $(U,\varphi) \in \mathcal{A}$ 
such that $m \in U$ and a chart $(W,\psi)\in \mathcal{B}$ such that $F(m)\in W$ where $F(U) \subset W,$ and the map 
\begin{equation*}
\psi \circ F \circ \varphi^{-1}:\varphi(U) \rightarrow \psi(W)
\end{equation*}
is a map of class  $\mathcal{C}^p$. If it is
a $\mathcal{C}^p$ diffeomorphism, then we 
say that $F$  is a \emph{$\mathcal{C}^p$ diffeomorphism between manifolds}.
We  say that $\psi \circ F \circ \varphi^{-1}$
is a representation of $F$  using a system of local coordinates given by the charts 
$(U,\varphi)$ and $(W,\psi).$
\end{definition}

\begin{definition}
Let $\mathbb{B}$ be a $\mathcal{C}^p$ manifold with atlas $\mathcal{A}=\{(U_{b}, \varphi_{b}):b \in \mathbb{B}\}$, and let $\mathbb{F}$ be a manifold. 
A $\mathcal{C}^p$  
fibre bundle $\mathbb{E}$ with base $\mathbb{B}$ and typical fibre $\mathbb{F}$ 
is a $\mathcal{C}^p$ manifold which is locally a product manifold, that is, there exists a  
surjective morphism $\pi: \mathbb{E} \longrightarrow \mathbb{B}$ such that for each $b \in  \mathbb{B}$ there is a $\mathcal{C}^p$  diffeomorphism between manifolds $$\chi_{b}:\pi^{-1}(U_{b}) \longrightarrow U_{b} \times \mathbb{F},$$ such that $p_{b} \circ \chi_{b} = \pi$ where $p_{b} : U_{b} \times \mathbb{F} \longrightarrow  U_{b}$ is the projection. 
For each $b \in \mathbb{B},$ $\pi^{-1}(b) = \mathbb{E}_b$ is called the fibre over $b.$ 
The $\mathcal{C}^p$ diffeomorphisms $\chi_{b}$ are called fibre
bundle charts. If $p = 0,$   $\mathbb{E},\mathbb{B}$ and $\mathbb{F}$ are only required to be topological spaces and $\{U_{b}:b \in \mathbb{B}\}$ 
 an open covering of $\mathbb{B}.$ In the case where $\mathbb{F}$ is a Lie group, we say that
$\mathbb{E}$ is a $\mathcal{C}^p$  \emph{principal bundle}, and if
$\mathbb{F}$ is a vector space, we  say that it is  a $\mathcal{C}^p$ \emph{vector bundle}.
\end{definition}

\begin{theorem}
The set $\mathcal{M}_r(\mathbb{R}^{k\times r})$ is an analytic principal
bundle with typical fibre $\mathrm{GL}_r$ and base $\mathbb{G}_r(\mathbb{R}^k)$, with a surjective 
 morphism between $\mathcal{M}_r(\mathbb{R}^{k\times r})$ and $\mathbb{G}_r(\mathbb{R}^k)$ given by  the map $\mathrm{col}_{k,r}$.
\end{theorem}

\begin{proof}
To show that it is an analytic principal bundle, we first observe that
$$
\mathrm{col}_{k,r}:(\mathcal{M}_r(\mathbb{R}^{k\times r}),\mathcal{B}_{k,r}) \longrightarrow (\mathbb{G}_r(\mathbb{R}^k),\mathcal{A}_{k,r})
$$
is a surjective morphism. Indeed, let $Z\in \mathcal{M}_r(\mathbb{R}^{k\times r})$ and $(\mathcal{V}_Z,\xi_Z) \in \mathcal{B}_{k,r}$ and $(\mathfrak{U}_Z,\varphi_Z) \in \mathcal{A}_{k,r}$.
Noting 
that   
$\mathrm{col}_{k,r}(YG)= \mathrm{col}_{k,r}(Y)$ for all $Y\in \mathcal{S}_Z$, we obtain that 
$\mathrm{col}_{k,r}(\mathcal{V}_Z) = \mathfrak{U}_Z$.
Moreover, a representation of $\mathrm{col}_{k,r}$ 
by using a system of local coordinates given by the charts is 
$$
(\varphi_Z \circ \mathrm{col}_{k,r} \circ \xi_Z^{-1})(X,G) = X,
$$
which is clearly  an analytic map from $\mathbb{R}^{(k-r)\times r} \times \mathrm{GL}_r$
to $\mathbb{R}^{(k-r)\times r}$ such that 
$
\mathrm{col}_{k,r}^{-1}(\mathfrak{U}_{Z}) =\mathcal{V}_Z.
$
Now, a representation of the morphism
$$
\chi_Z: (\mathcal{V}_Z,\{(\mathcal{V}_Z,\xi_Z)\}) \longrightarrow  (\mathfrak{U}_{Z},\{(\mathfrak{U}_{Z},\varphi_Z)\}) \times 
(\mathrm{GL}_r,\{(\mathrm{GL}_r,id_{\mathbb{R}^{r\times r}})\}),\quad
W\mapsto (\mathrm{col}_{k,r}(W),G)
$$
using the system of local coordinates given by the charts is 
$$
((\varphi_Z \times id_{\mathbb{R}^{r\times r}}) \circ \chi_Z \circ \xi_Z^{-1})
:\mathbb{R}^{(k-r) \times r} \times \mathrm{GL}_r \longrightarrow \mathbb{R}^{(k-r) \times r} \times \mathrm{GL}_r,
$$
defined by
$$
((\varphi_Z \times id_{\mathbb{R}^{r\times r}}) \circ \chi_Z \circ \xi_Z^{-1})(X,G) = (X,G),
$$
which is clearly an analytic diffeomorphism. To conclude, consider the projection
$$
p_Z: \mathfrak{U}_{Z} \times \mathrm{GL}_r \longrightarrow \mathfrak{U}_{Z}, \quad
(\mathfrak{V},G) \mapsto \mathfrak{V},
$$
and observe that $(p_{Z} \circ \chi_{Z})(W) = \mathrm{col}_{k,r}(W)$ holds for all $W \in \mathcal{V}_Z.$ 
\end{proof}

\subsection{$\mathcal{M}_r(\mathbb{R}^{k\times r})$ as a submanifold and its tangent space}

Here, we prove that the  non-compact Stiefel 
manifold $\mathcal{M}_r(\mathbb{R}^{k\times r})$
equipped with the topology given by the atlas $\mathcal{B}_{k,r}$
is an embedded submanifold in $\mathbb{R}^{k\times r}$. For that, we have to prove that
the standard inclusion map 
$$
i:(\mathcal{M}_r(\mathbb{R}^{k\times r}),\mathcal{B}_{k,r}) \longrightarrow (\mathbb{R}^{k\times r},\{(\mathbb{R}^{k\times r},id_{\mathbb{R}^{k\times r}})\})
$$
as a morphism is an embedding.  To see this we need to recall some definitions and results.

\begin{definition}
Let $F:(\mathbb{M},\mathcal{A})\rightarrow (\mathbb{N},\mathcal{B})$ 
be a morphism between $\mathcal{C}^p$ manifolds and let $m\in
\mathbb{M}.$ We say that $F$ is an \emph{immersion at $m$} if there exists an
open neighbourhood $U_m$ of $m$ in $\mathbb{M}$ such that the restriction of $F$ to 
$U_m$ induces an isomorphism from $U_m$ onto a submanifold of $\mathbb{N}.$ We
say that $F$ is an \emph{immersion} if it is an immersion at each point of $%
\mathbb{M}.$
\end{definition}

 The next step is to recall the definition of the differential as a morphism
which gives a linear map between the tangent spaces of the manifolds (in local coordinates)
involved with the morphism.  Let us recall that for any $m \in \mathbb{M}$, we denote by $\mathbb{T}_m \mathbb{M}$ the tangent space of $\mathbb{M}$ at $m$ (in local coordinates).

\begin{definition}
Let $(\mathbb{M},\mathcal{A})$ and $(\mathbb{N},\mathcal{B})$ be two $
\mathcal{C}^{p}$ manifolds. 
Let $F:(\mathbb{M},\mathcal{A})\rightarrow (\mathbb{N},\mathcal{B})$ be a  morphism of class $
\mathcal{C}^{p}$, i.e., for any $m \in \mathbb{M}$,
\begin{equation*}
\psi \circ F\circ \varphi ^{-1}:\varphi (U)\rightarrow \psi (W)
\end{equation*}%
is a map of class $\mathcal{C}^{p}$, where $(U,\varphi )\in \mathcal{A}$
is a chart in $\mathbb{M}$ containing $m$ and $(W,\psi )\in \mathcal{B}$ is a chart in $\mathbb{N}$ containing $F(m)$. Then %
 we define 
\begin{equation*}
\mathrm{T}_{m}F:\mathbb{T}_{m}(\mathbb{M})\longrightarrow \mathbb{T}_{F(m)}(\mathbb{N}),\quad
\upsilon \mapsto  D(\psi \circ F\circ \varphi ^{-1})(\varphi (m))[\upsilon].
\end{equation*}
\end{definition}

For finite dimensional manifolds we have the following criterion for immersions (see
Theorem 3.5.7 in \cite{abraham2012}).

\begin{proposition}
\label{prop_inmersion} Let $(\mathbb{M},\mathcal{A})$ and $(\mathbb{N},\mathcal{B})$ 
be $\mathcal{C}^{p}$ manifolds. Let $$F:(\mathbb{M},\mathcal{A})\rightarrow (\mathbb{N},\mathcal{B})$$ be a $\mathcal{C}^{p}$ morphism
and $m\in \mathbb{M}.$ Then $F$ is an immersion at $m$ if and only if $\mathrm{T}_{m}F $ is injective.
\end{proposition}

A concept related to an immersion between manifolds is given in the
following definition.

\begin{definition}
Let $(\mathbb{M},\mathcal{A})$ and $(\mathbb{N},\mathcal{B})$ 
be $\mathcal{C}^{p}$ manifolds and let $f:(\mathbb{M},\mathcal{A})\longrightarrow (\mathbb{N},\mathcal{B})$
be a $\mathcal{C}^{p}$ morphism. If $f$ is an injective immersion, then $%
f(\mathbb{M}) $ is called an \emph{immersed submanifold of $\mathbb{N}$.}
\end{definition}

Finally, we give the definition of embedding.

\begin{definition}
Let $(\mathbb{M},\mathcal{A})$ and $(\mathbb{N},\mathcal{B})$ 
be $\mathcal{C}^{p}$ manifolds and let $f:(\mathbb{M},\mathcal{A})\longrightarrow (\mathbb{N},\mathcal{B})$
be a $\mathcal{C}^{p}$ morphism. If $f$ is an injective immersion, and 
$f:(\mathbb{M},\tau_{\mathcal{A}}) \longrightarrow (f(\mathbb{M}),\tau_{\mathcal{B}}|_{f(\mathbb{M})})$
is a topological homeomorphism, then we say that $f$ is an \emph{embedding} and
$f(\mathbb{M})$ is called an \emph{embedded submanifold of $\mathbb{N}$.}
\end{definition}

We first note that the representation of the inclusion map $i$  using the system of local coordinates given by the charts $(\mathcal{V}_Z,\xi_Z) \in \mathcal{B}_{k,r}$
in $\mathcal{M}_r(\mathbb{R}^{k\times r})$ and $(\mathbb{R}^{k\times r},id_{\mathbb{R}^{k\times r}})$ in $\mathbb{R}^{k\times r}$
is  
$$
 (id_{\mathbb{R}^{k\times r}} \circ i \circ \xi_{Z}^{-1}) = (i \circ \xi_{Z}^{-1}) :\mathbb{R}^{(k-r)\times r} \times \mathrm{GL}_r \rightarrow \mathbb{R}^{k\times r}, \quad (X,G) \mapsto (Z+Z_{\bot}X)G.
$$
Then the tangent map $T_Z i$ at $Z= \xi_Z^{-1}(0,id_r)$, defined by $T_Z i = D (i \circ \xi_{Z}^{-1})(0,id_r)$, is   
\begin{align*}
T_Z i : \mathbb{R}^{(k-r)\times r} \times \mathbb{R}^{r\times r} \to \mathbb{R}^{k\times r},\quad 
 (\dot X,\dot G) \mapsto Z_{\bot} \dot X  + Z \dot G.
\end{align*}

\begin{proposition}\label{fullrank-submanifold}
The tangent map $\mathrm{T}_{Z}i:\mathbb{R}^{(k-r)\times r} \times \mathbb{R}^{r \times r} \rightarrow \mathbb{R}^{k\times r}$ at $Z \in \mathcal{M}_r(\mathbb{R}^{k\times r})$ is a linear isomorphism, with inverse  $(\mathrm{T}_{Z}i)^{-1}$ given by 
  
$$(\mathrm{T}_{Z}i)^{-1}(\dot  Z) = (Z_{\bot}^+\dot  Z,Z^+ \dot  Z),$$
for $\dot  Z \in  \mathbb{R}^{k\times r}$.
Furthermore, the standard inclusion map $i$  is an embedding from $\mathcal{M}_r(\mathbb{R}^{k\times r})$ to $\mathbb{R}^{k\times r}.$
\end{proposition}

\begin{proof}
Let us assume that
$
\mathrm{T}_{Z}i(\dot  X,\dot  G) = Z_{\bot}\dot  X + Z\dot  G = 0.
$
Multiplying this equality by $Z^+$ and $Z_\bot^+$ on the left, we obtain $\dot G=0$ and $\dot X=0$ respectively, which implies that  
$\mathrm{T}_{Z}i$ is injective. To prove that it is
also surjective, we consider a matrix $\dot Z \in \mathbb{R}^{k \times r}$ and observe that $\dot X = Z_{\bot}^+\dot  Z \in \mathbb{R}^{(k-r)\times r} $ and $\dot G = Z^+ \dot  Z \in  \mathbb{R}^{r \times r}$ is such that 
$\mathrm{T}_{Z}i(\dot  X,\dot  G) = \dot Z$.
Since $\mathrm{T}_{Z}i$ is injective, the inclusion map $i$ is an immersion.\\
To prove that it is an
embedding we equip $\mathcal{M}_r(\mathbb{R}^{k\times r})$ with the topology $\tau_{\mathcal{B}_{k,r}}$ given by the atlas  and we equip $\mathbb{R}^{k \times r}$  with the topology $\tau_{\mathbb{R}^{k\times r}}$ induced by matrix norms.
We need to check that 
$$
i:(\mathcal{M}_{r}(\mathbb{R}^{k\times r}),\tau_{\mathcal{B}_{k,r}})
\longrightarrow (\mathcal{M}_{r}(\mathbb{R}^{k\times r}),\tau_{\mathbb{R}^{k\times r}}|_{\mathcal{M}_{r}(\mathbb{R}^{k\times r})})
$$
is a topological homeomorphism. Since the topology in $(\mathcal{M}_{r}(\mathbb{R}^{k\times r}),\tau_{\mathcal{B}_{k,r}})$
has the property that each local chart $\xi_Z$ is indeed a homeomorphism from $\mathcal{V}_Z$ in $\mathcal{M}_{r}(\mathbb{R}^{k\times r})$ to $\xi_Z(\mathcal{V}_Z) = \mathbb{R}^{(k-r)\times r} \times \mathrm{GL}_r$ (see Section \ref{subsec:elements_geometry}), we only need to show that the bijection
$(i \circ \xi_{Z}^{-1}):\mathbb{R}^{(k-r)\times r} \times \mathrm{GL}_r \rightarrow \mathcal{V}_Z \subset \mathbb{R}^{k\times r}$ given by
$$(i \circ \xi_{Z}^{-1})(X,G) = (Z+Z_{\bot}X)G$$ 
is a topological homeomorphism for all $Z \in \mathcal{M}_{r}(\mathbb{R}^{k\times r}).$ 
Observe that $D(i \circ \xi_{Z}^{-1})(X,G) \in \mathcal{L}(\mathbb{R}^{(k-r)\times r} \times \mathbb{R}^{r \times r},\mathbb{R}^{k\times r})$ is given by
$$
D(i \circ \xi_{Z}^{-1})(X,G)[(\dot{X},\dot{G})] = Z_{\bot}\dot{X}G + (Z+Z_{\bot}X)\dot{G}.
$$
Assume that $Z_{\bot}\dot{X}G + (Z+Z_{\bot}X)\dot{G} = 0.$ Multiplying 
this equality by $Z^+$ on the left we obtain $\dot{G} = 0,$ and hence
$Z_{\bot}\dot{X}G = 0.$ Multiplying by $Z_{\bot}^+$ on the left 
we obtain $\dot{X} G= 0.$ Thus $\dot{X}=0$ and 
as a consequence $D(i \circ \xi_{Z}^{-1})(X,G)$ is a linear isomorphism 
for each $(X,G) \in \mathbb{R}^{(k-r)\times r} \times \mathrm{GL}_r.$ The inverse function
theorem says us that $(i \circ \xi_{Z}^{-1})$ is a diffeomorphism, in particular a homeomorphism,
and hence $i$ is an embedding.
\end{proof}

The tangent space to $\mathcal{M}_r(\mathbb{R}^{k\times r})$ at $Z$ is the image through  $T_Z i $ of the tangent space at $Z$ in local coordinates $\mathbb{T}_Z \mathcal{M}_r(\mathbb{R}^{k\times r})  = \mathbb{R}^{(k-r)\times r} \times  \mathbb{R}^{r\times r}$, i.e. 
$${T}_Z \mathcal{M}_r(\mathbb{R}^{k\times r}) =  \{Z_{\bot} \dot X  + Z \dot G :\dot X \in \mathbb{R}^{(k-r)\times r} ,\dot G \in  \mathbb{R}^{r\times r}  \} =  \mathbb{R}^{k\times r},$$
and can be decomposed into a vertical tangent space 
$$
T_Z^V\mathcal{M}_r(\mathbb{R}^{k\times r}) = \{  Z \dot G :\dot G \in \mathbb{R}^{r\times r} \},
$$
and an horizontal tangent space
$$
T^H_Z\mathcal{M}_r(\mathbb{R}^{k\times r}) = \{Z_{\bot} \dot X    :\dot X \in \mathbb{R}^{(k-r)\times r}  \}.
$$

\subsection{Lie group structure of neighbourhoods $\mathcal{V}_Z$}

We here prove that each neighbourhood $\mathcal{V}_Z$ of $\mathcal{M}_{r}(\mathbb{R}^{k\times r})$ has the structure of a Lie group. For that, we first note that 
$\mathcal{V}_Z$ can be identified with 
$\mathcal{S}_{Z} \times \mathrm{GL}_r$, with $\mathcal{S}_{Z}$ given by \eqref{SZLiegroup}. Noting that $\mathcal{S}_{Z}$ can be identified with the Lie group $\mathcal{G}_Z$ defined in \eqref{GZ}, we then have that $\mathcal{V}_Z$ can be identified with a product of two Lie groups $\mathcal{G}_Z \times \mathrm{GL}_r$, which is a Lie group with the  
group operation $\odot_Z$ given by
$$
(\exp(Z_{\bot}XZ^+),G) \odot_Z (\exp(Z_{\bot} X'Z^+), G') = (\exp(Z_{\bot}(X+ X')Z^+),G G'), 
$$
for $X, X'\in \mathbb{R}^{(k-r)\times r}$ and $G, G'\in \mathrm{GL}_r$. 
It allows us to define a group operation $\star_Z$ over $\mathcal{V}_Z$ defined 
for $W =\xi_{Z}^{-1}(X,G) $ and $ W' =  \xi_{Z}^{-1}( X', G') $ by
\begin{align}
W \star_Z  W' = \xi_{Z}^{-1}(X+ X',G G'),\label{group operation VZ}
\end{align}
and to state the following result. 

\begin{theorem}
The set $\mathcal{V}_Z$ together with the group operation $\star_Z$ defined by \eqref{group operation VZ} 
is a Lie group and the map
$
\eta_{Z}:\mathcal{V}_Z \longrightarrow \mathcal{G}_{Z} \times \mathrm{GL}_r
$
given by
$$
\eta_{Z}(\xi_{Z}^{-1}(X,G))= (\exp(Z_{\bot}XZ^+),G)
$$
is a Lie group isomorphism.
\end{theorem}

 
\section{The principal bundle $\mathcal{M}_r(\mathbb{R}^{n\times m})$ for $0 < r < \min(m,n)$}\label{sec:fixed-rank}

In this section, we give a geometric description of the set of matrices $\mathcal{M}_r(\mathbb{R}^{n\times m})$ with rank $r<\min(m,n)$.
 
\subsection{$\mathcal{M}_r(\mathbb{R}^{n\times m})$ as a principal bundle}

For $Z\in \mathcal{M}_r(\mathbb{R}^{n\times m})$, there exists 
$U \in \mathcal{M}_r(\mathbb{R}^{n\times r}),$ $V \in \mathcal{M}_r(\mathbb{R}^{m\times r}),$
and $G \in \mathrm{GL}_r$ such that
$$
Z = UGV^T,
$$
where the column space of $Z$ is $\mathrm{col}_{n,r}(U)$ and the row space of $Z$ is $\mathrm{col}_{m,r}(V).$

Let us first introduce the surjective map
$$
\varrho_{r}: \mathcal{M}_r(\mathbb{R}^{n \times m}) \longrightarrow
\mathbb{G}_r(\mathbb{R}^n) \times \mathbb{G}_r(\mathbb{R}^m), \quad UGV^T \mapsto (\mathrm{col}_{n,r}(U),\mathrm{col}_{m,r}\,(V)).
$$
The set
$$
\varrho_{r}^{-1}(\mathrm{col}_{n,r}(U),\mathrm{col}_{m,r}\,(V))=
\{UHV^T: H \in \mathrm{GL}_r\} 
$$
can be identified with $\mathrm{GL}_r$. Let us consider 
$U_{\bot} \in \mathcal{M}_{n-r}(\mathbb{R}^{n\times (n-r)})$ such that $U^T \, U_{\bot} = 0$ and
$V_{\bot} \in \mathcal{M}_{m-r}(\mathbb{R}^{m\times (m-r)})$  such that $V^T \, V_{\bot} = 0.$  Then we define a neighbourhood of $UGV^T$ 
in the set $\mathcal{M}_r(\mathbb{R}^{n \times m})$ by $$\mathcal{U}_Z:=
\varrho_r^{-1}(\mathfrak{U}_U \times \mathfrak{U}_V),$$
where $\mathfrak{U}_U$ and $\mathfrak{U}_V$ are the neighbourhoods of  $\mathrm{col}_{n,r}(U)$ and $\mathrm{col}_{m,r}(V)$ respectively (see Section \ref{sec:liegroupSZ}). 
Noting that $\mathfrak{U}_U = \varphi^{-1}_U(\mathbb{R}^{(n-r)\times r}) = \mathrm{col}_{n,r}(\mathcal{S}_U)$ and $\mathfrak{U}_V = \varphi_V^{-1}(\mathbb{R}^{(m-r)\times r}) = \mathrm{col}_{m,r}(\mathcal{S}_V)$, where $\mathcal{S}_U$ and $\mathcal{S}_V$ are the affine cross sections of $U$ and $V$ respectively (defined by \eqref{SZ}), the neighbourhood of $UGV^T$ can be written
$$
\mathcal{U}_Z=\{
(U+U_{\bot}X)H(V+V_{\bot}Y)^T: (X,Y,H) \in \mathbb{R}^{(n-r)\times r} \times \mathbb{R}^{(m-r)\times r} \times \mathrm{GL}_r
\}.
$$
We can associate to $\mathcal{U}_Z$ the parametrisation $\theta_Z^{-1}$ given by the chart  (see Figure \ref{Mnxm})
$$
\theta_Z:  \mathcal{U}_Z \rightarrow \mathbb{R}^{(n-r)\times r} \times \mathbb{R}^{(m-r)\times r} \times \mathrm{GL}_r
$$
defined by 
$$
\theta_Z^{-1}(X,Y,H) = (U+U_{\bot}X)H(V+V_{\bot}Y)^T 
$$
for $(X,Y,H)\in  \mathbb{R}^{(n-r)\times r} \times \mathbb{R}^{(m-r)\times r} \times \mathrm{GL}_r$, and 
$$
\theta_Z(A) = (U_\bot^+A(V^+)^T(U^+A(V^+)^T)^{-1},V_\bot^+A^T(U^+)^T(V^+A^T(U^+)^T)^{-1},U^+A(V^+)^T)
$$
for $A\in \mathcal{U}_Z$.
In particular, we have $\theta_Z^{-1}(0,0,G) = Z.$ We  point out that 
$\mathcal{U}_Z = \mathcal{U}_{Z'}$ and $\theta_Z=\theta_{Z'}$ for every $Z'=UG'V^T$ with $G'\neq G.$

\begin{figure}[h]
\centering
\includegraphics[scale=1]{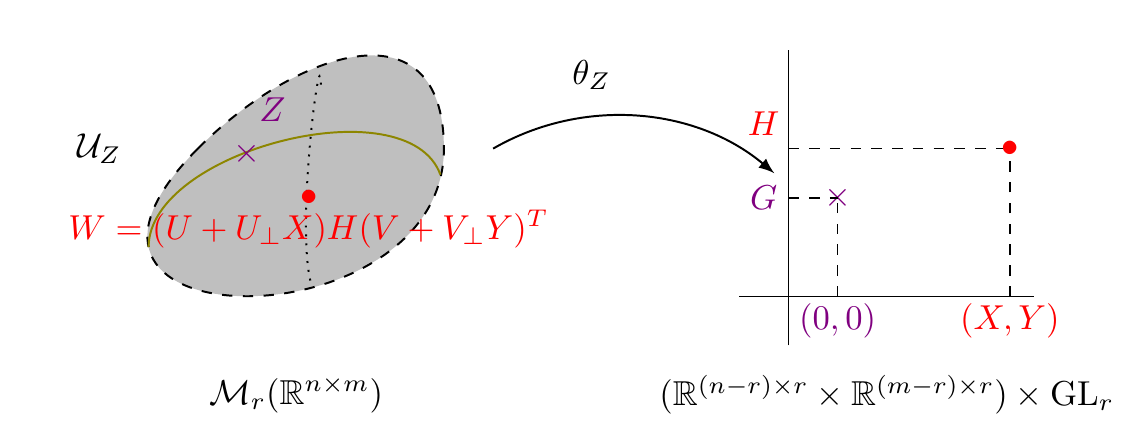}
\caption{Illustration of the chart $\theta_Z$ which associates to $W=  (U+U_{\bot}X)H(V+V_{\bot}Y)^T \in \mathcal{U}_Z \subset \mathcal{M}_{r}(\mathbb{R}^{n\times m})$, the parameters $(X,Y,G)$ in $\mathbb{R}^{(n-r) \times r}\times \mathbb{R}^{(m-r) \times r} \times \mathrm{GL}_r$.}\label{Mnxm}
\end{figure}

\begin{theorem}
The collection $\mathcal{B}_{n,m,r}:=\{(\mathcal{U}_Z,\theta_Z): Z \in  \mathcal{M}_r(\mathbb{R}^{n \times m})\}$
 is
an analytic atlas for $\mathcal{M}_r(\mathbb{R}^{n \times m})$ and hence $(\mathcal{M}_r(\mathbb{R}^{n \times m}),\mathcal{B}_{n,m,r})$ 
is an analytic $r(n+m-r)$-dimensional manifold modelled on $\mathbb{R}^{(n-r) \times r}\times \mathbb{R}^{(m-r) \times r} \times \mathbb{R}^{r \times r}.$
\end{theorem}

\begin{proof}
$\{\mathcal{U}_Z\}_{Z \in \mathcal{M}_r(\mathbb{R}^{n\times m})}$ is clearly a covering
of $\mathcal{M}_r(\mathbb{R}^{n\times m})$. Moreover, 
since $\theta_{Z}$ is bijective from $\mathcal{U}_Z$ to $\mathbb{R}^{(n-r) \times r} \times \mathbb{R}^{(m-r) \times r} \times \mathrm{GL}_r$, we claim that if $\mathcal{U}_Z \cap \mathcal{U}_{\tilde Z} \neq \emptyset$ for $Z = UGV^T$ and $\tilde Z = \tilde U\tilde G\tilde V^T \in \mathcal{M}_r(\mathbb{R}^{n\times m}),$ then  the following statements hold:
\begin{enumerate}
\item[i)] $\theta_{Z}(\mathcal{U}_Z \cap \mathcal{U}_{\tilde Z})$ and $\theta_{\tilde Z}(\mathcal{U}_Z \cap \mathcal{U}_{\tilde Z})$ are open sets
in $\mathbb{R}^{(n-r) \times r} \times \mathbb{R}^{(m-r) \times r} \times \mathrm{GL}_r$ and
\item[ii)]the map $
\theta_{\tilde Z} \circ \theta_{Z}^{-1}
$ is analytic from $\theta_{Z}(\mathcal{U}_Z \cap \mathcal{U}_{\tilde Z}) \subset
\mathbb{R}^{(n-r) \times r} \times \mathbb{R}^{(m-r) \times r} \times \mathrm{GL}_r$ to $\theta_{\tilde Z}(\mathcal{U}_Z \cap \mathcal{U}_{\tilde Z}) \subset \mathbb{R}^{(n-r) \times r} \times \mathbb{R}^{(m-r) \times r} \times \mathrm{GL}_r$.
\end{enumerate}
In this proof, we equip $\mathbb{R}^{n\times m}$ with the topology $\tau_{\mathbb{R}^{n\times m}}$ induced by matrix norms. 
We first observe that the set 
$\mathcal{U}_Z = \{A \in \mathcal{M}_{r}(\mathbb{R}^{n\times m}) : \det(U^TA V)\neq 0\} = \mathcal{O}_Z \cap  \mathcal{M}_{r}(\mathbb{R}^{n\times m}) $, where $\mathcal{O}_Z = \{A \in \mathbb{R}^{n\times m} : \det(U^TA V)\neq 0\}$, as the inverse image of the open set $\mathbb{R}\setminus \{0\}$ through the continuous map $A\mapsto \det(U^TAV)$ from $\mathbb{R}^{n\times m}$ to $\mathbb{R}$, 
 is an open set in $\mathbb{R}^{n\times m}$.   In the same way, we have that $\mathcal{U}_{\tilde Z} = \mathcal{O}_{\tilde Z} \cap  \mathcal{M}_{r}(\mathbb{R}^{n\times m})$, with $\mathcal{U}_{\tilde Z} $ an open set in  $\mathbb{R}^{n\times m}$.
Since  $\mathcal{U}_Z \cap \mathcal{U}_{\tilde Z} = \mathcal{O}_Z \cap  \mathcal{O}_{\tilde Z} \cap \mathcal{M}_{r}(\mathbb{R}^{n\times m}),$ and since the image of $\theta_Z^{-1}$ is in $\mathcal{M}_{r}(\mathbb{R}^{n\times m})$, we have 
 $$\theta_Z(\mathcal{U}_Z \cap \mathcal{U}_{\tilde Z}) = (\theta_Z^{-1})^{-1}(\mathcal{U}_Z \cap \mathcal{U}_{\tilde Z}) = (\theta_Z^{-1})^{-1}(\mathcal{O}_Z \cap  \mathcal{O}_{\tilde Z}),$$
 the inverse image through $\theta_Z^{-1}$ of the open set $\mathcal{O}_Z \cap  \mathcal{O}_{\tilde Z}$ in $\mathbb{R}^{n\times m}$. Since $\theta_Z^{-1}$ is a continuous map from    $\mathbb{R}^{(n-r) \times r} \times \mathbb{R}^{(m-r) \times r} \times \mathrm{GL}_r$ to $\mathbb{R}^{n\times m}$, we deduce that $\theta_Z(\mathcal{U}_Z \cap \mathcal{U}_{\tilde Z})$ is an open set in $\mathbb{R}^{(n-r) \times r} \times \mathbb{R}^{(m-r) \times r} \times \mathrm{GL}_r$. 
 Similarly, $\theta_{\tilde Z}(\mathcal{U}_Z \cap \mathcal{U}_{\tilde Z})$ is an open set in $\mathbb{R}^{(n-r) \times r} \times \mathbb{R}^{(m-r) \times r} \times \mathrm{GL}_r$.
Now, let $(X,Y,H) \in \mathbb{R}^{(n-r) \times r} \times \mathbb{R}^{(m-r) \times r} \times \mathrm{GL}_r$ such that  
$ \theta_{Z}^{-1}(X,Y,H)   \in \mathcal{U}_Z \cap \mathcal{U}_{\tilde Z}$. From the expressions of $\theta_Z^{-1}$ and $ \theta_{\tilde Z}$, the map $\theta_{\tilde Z}\circ \theta_{Z}^{-1}$ is defined by
\begin{align*}
\theta_{\tilde Z}\circ \theta_{Z}^{-1}(X,Y,H)=  (&\tilde U_\bot^+\theta_{Z}^{-1}(X,Y,H)(\tilde V^+)^T(\tilde U^+\theta_{Z}^{-1}(X,Y,H)(\tilde V^+)^T)^{-1},\\
&\tilde V_\bot^+\theta_{Z}^{-1}(X,Y,H)^T(\tilde U^+)^T(\tilde V^+\theta_{Z}^{-1}(X,Y,H)^T(\tilde U^+)^T)^{-1},\\
&\tilde U^+\theta_{Z}^{-1}(X,Y,H)(\tilde V^+)^T),
\end{align*}
with $\theta_{Z}^{-1}(X,Y,H) = (U+U_\bot X)H(V+V_\bot Y)^T$, which is clearly an analytic map.
 \end{proof}

\begin{theorem}
The set $\mathcal{M}_r(\mathbb{R}^{n \times m})$ is an analytic principal bundle with typical fibre $\mathrm{GL}_r$ and base $\mathbb{G}_r(\mathbb{R}^n) \times \mathbb{G}_r(\mathbb{R}^m)$ with surjective morphism $\varrho_{r}$ between $\mathcal{M}_r(\mathbb{R}^{n \times m})$  and  $\mathbb{G}_r(\mathbb{R}^n) \times \mathbb{G}_r(\mathbb{R}^m)$ given by $\varrho_r.$
\end{theorem}

\begin{proof}
To prove that it is an analytic principal bundle, we consider the surjective map
$$
\varrho_{r}:\mathcal{M}_r(\mathbb{R}^{n \times m})
\longrightarrow \mathbb{G}_r(\mathbb{R}^n) \times \mathbb{G}_r(\mathbb{R}^m), \quad
UGV^T \mapsto (\mathrm{col}_{n,r} (U), \mathrm{col}_{m,r}(V)),
$$
the atlas $\mathcal{A}_{n,r}:=\{(\mathfrak{U}_{U},\varphi_{U}):U \in \mathcal{M}_{r}(\mathbb{R}^{n\times r})\}$ of $\mathbb{G}_r(\mathbb{R}^n)$ and 
the atlas $\mathcal{A}_{m,r}:=\{(\mathfrak{U}_{V},\varphi_{V}):V \in \mathcal{M}_{r}(\mathbb{R}^{m\times r})\}$ of $\mathbb{G}_r(\mathbb{R}^m).$ Recall that
$$
\mathfrak{U}_{Z}=\{\mathrm{col}_{k,r}(Z+Z_{\bot}X): X\in \mathbb{R}^{(k-r)\times r}\},
$$
with $k=n$ if $Z=U$ or $k=m$ if $Z=V$, and hence
$$
\varrho_{r}^{-1}(\mathfrak{U}_{U},\mathfrak{U}_{V})=
\left\{(U+U_{\bot}X)H(V+V_{\bot}Y)^T:
X \in \mathbb{R}^{(n-r)\times r}, 
Y \in \mathbb{R}^{(m-r)\times r},  H \in  \mathrm{GL}_r
\right\}.
$$
Observe that for each fixed $G \in \mathrm{GL}_r$, we have that
$\varrho_{r}^{-1}(\mathfrak{U}_{U},\mathfrak{U}_{V})= \mathcal{U}_Z,$
where $Z=UGV^T$. Since $\mathcal{U}_Z = \mathcal{U}_{Z'}$ holds for $Z'=UG'V^T,$ 
where $G' \in \mathrm{GL}_r,$ the map
$$
\chi_Z: \mathcal{U}_Z \longrightarrow \mathfrak{U}_{U} \times \mathfrak{U}_{V} \times \mathrm{GL}_r
$$
defined by
$$
\chi_Z(U'H'(V')^T):=(\mathrm{col}_{n,r}(U'),\mathrm{col}_{m,r}(V'),H'),
$$
is independent of the choice of $Z=UGV^T,$ where $G \in \mathrm{GL}_r.$ Now, the representation of $\chi_Z$ in local coordinates
is the map
$$
((\varphi_U \times \varphi_V \times id_{\mathbb{R}^{r\times r}})\circ \chi_Z \circ \theta_Z^{-1}):
\mathbb{R}^{(n-r) \times r}\times \mathbb{R}^{(m-r) \times r} \times \mathrm{GL}_r 
\longrightarrow \mathbb{R}^{(n-r) \times r}\times \mathbb{R}^{(m-r) \times r} \times \mathrm{GL}_r
$$
given by $((\varphi_U \times \varphi_V \times id_{\mathbb{R}^{r\times r}})\circ \chi_Z \circ \theta_Z^{-1})(X,Y,H)=(X,Y,H)$, which
is an analytic diffeomorphism. Moreover, let $p_Z:\mathfrak{U}_{U} \times \mathfrak{U}_{V} \times \mathrm{GL}_r \longrightarrow \mathfrak{U}_{U} \times \mathfrak{U}_{V}$ be the projection over the two first components. Then
$$
(p_Z \circ \chi_Z)(UHV^T)=(\mathrm{col}_{n,r}(U),\mathrm{col}_{m,r}(V)) = \varrho_{r}(UHV^T)
$$
and the theorem follows.
\end{proof}

\subsection{$\mathcal{M}_r(\mathbb{R}^{n\times m})$ as a submanifold and its tangent space}

Here, we prove that $\mathcal{M}_r(\mathbb{R}^{n\times m})$  equipped with the topology given by the atlas
$\mathcal{B}_{n,m,r}$ is an embedded submanifold in $\mathbb{R}^{n\times m}.$ 
For that, we have to prove that the standard inclusion map ${i}:\mathcal{M}_r(\mathbb{R}^{n \times m}) \rightarrow
\mathbb{R}^{n \times m}$ is an embedding. Noting that the inclusion map restricted to the neighbourhood $\mathcal{U}_Z$ of $Z = UGV^T$ is identified with 
$$
({i} \circ \theta_{Z}^{-1}):\mathbb{R}^{(n-r)\times r} \times \mathbb{R}^{(m-r)\times r} \times \mathrm{GL}_r \longrightarrow \mathbb{R}^{n \times m}, \quad (X,Y,H) \mapsto (U+U_{\bot}X)H(V+V_{\bot}Y)^T,
$$
the tangent map $T_Z i$ at $Z = \theta_Z^{-1}(0,0,G)$, defined by $T_Zi = D({i} \circ \theta_{Z}^{-1})(0,0,G)$, is 
\begin{align*}
\mathrm{T}_{Z}{i}:\mathbb{R}^{(n-r)\times r} \times \mathbb{R}^{(m-r)\times r} \times 
\mathbb{R}^{r\times r} \rightarrow \mathbb{R}^{n \times m}, \quad (\dot  X,\dot  Y,\dot  H) \mapsto 
U_{\bot}\dot  XGV^T  + U G (V_{\bot}\dot  Y)^T + U\dot  H V^T.
\end{align*}

\begin{proposition}\label{fixedrank-submanifold}
The tangent map $\mathrm{T}_{Z}i:\mathbb{R}^{(n-r)\times r} \times\mathbb{R}^{(m-r)\times r} \times \mathbb{R}^{r \times r} \rightarrow \mathbb{R}^{n\times m}$ at $Z = UGV^T \in \mathcal{M}_r(\mathbb{R}^{n\times m})$ is a linear isomorphism with inverse $ (T_Z i)^{-1}$ given by 
  $$
 (T_Z i)^{-1}(\dot Z) = (U_\bot^+ \dot Z (V^+)^TG^{-1},V_\bot^+\dot Z^T (U^+)^TG^{-T},U^+ \dot Z(V^+)^T)
 $$ 
 for $\dot Z \in \mathbb{R}^{n \times m} $. 
Furthermore, the standard inclusion map $i$ is an embedding from $\mathcal{M}_r(\mathbb{R}^{n \times m})$
to $\mathbb{R}^{n \times m}.$
\end{proposition}

\begin{proof}
Let us suppose that 
$
\mathrm{T}_{Z}i(\dot  X, \dot Y, \dot  H) =0.
$
Multiplying this equality by $(U_\bot)^+$ and $U^+$ on the left leads to  
$$
\dot X G V^T =0 \text{ and } G(V_\bot \dot Y)^T + \dot H V^T =0
$$
respectively. 
By multiplying the first equation by $(V^+)^T$ on the right, we obtain $\dot X =0$. By multiplying 
the second equation on the right by $(V^+)^T$ and  $(V_\bot^+)^T$, we respectively obtain $\dot H =0$ and $\dot Y=0$.  
Then, $\mathrm{T}_{Z}i$ is injective and then $i$ is an immersion. For $\dot Z \in \mathbb{R}^{n \times m}$, we note  that  $\dot X = U_\perp^+ \dot Z(V^+)^T G^{-1} \in \mathbb{R}^{n \times r}$, $\dot Y = V_\perp^+ \dot Z^T(U^+)^T G^{-T} \in \mathbb{R}^{m \times r}$, and $\dot G = U^+ \dot Z (V^+)^T \in \mathbb{R}^{r \times r}$ is such that $T_Zi(\dot X,\dot Y,\dot G) = \dot Z$, then $\mathrm{T}_{Z}i$ is also surjective. 
Let us now equip $\mathcal{M}_r(\mathbb{R}^{n\times m})$ with the topology $\tau_{\mathcal{B}_{n,m,r}}$ given by the atlas and $\mathbb{R}^{n \times m}$ with the topology $\tau_{\mathbb{R}^{n\times m}}$ induced by matrix norms. We have to prove that 
$$
i : (\mathcal{M}_r(\mathbb{R}^{n\times m}), \tau_{\mathcal{B}_{n,m,r}}) \longrightarrow (\mathcal{M}_r(\mathbb{R}^{n\times m}), \tau_{\mathbb{R}^{n\times m}|\mathcal{M}_r(\mathbb{R}^{n\times m})})
$$
is a topological isomorphism. The topology in $(\mathcal{M}_{r}(\mathbb{R}^{n\times m}),\tau_{\mathcal{B}_{n,m,r}})$
is such that a local chart $\theta_Z$ is a homeomorphism from $\mathcal{U}_Z \subset \mathcal{M}_{r}(\mathbb{R}^{n\times m})$ to $\theta_Z(\mathcal{U}_Z) = \mathbb{R}^{(n-r)\times r} \times \mathbb{R}^{(m-r)\times r} \times \mathrm{GL}_r$ (see Section \ref{subsec:elements_geometry}).
Then, to prove that the map $i$ is an embedding, we need to show that the bijection $$({i} \circ \theta_{Z}^{-1}):\mathbb{R}^{(n-r)\times r} \times \mathbb{R}^{(m-r)\times r} \times \mathrm{GL}_r \longrightarrow \mathcal{U}_Z \subset \mathbb{R}^{n \times m}$$ is a topological homeomorphism. For that, observe that its differential
$$
D({i} \circ \theta_{Z}^{-1})(X,Y,H) \in \mathcal{L}(\mathbb{R}^{(n-r)\times r} \times \mathbb{R}^{(m-r)\times r} \times\mathbb{R}^{r \times r}, \mathbb{R}^{n \times m})
$$
at $(X,Y,H) \in \mathbb{R}^{(n-r)\times r} \times \mathbb{R}^{(m-r)\times r} \times \mathrm{GL}_r$ is given by
\begin{align*}
&D({i} \circ \theta_{Z}^{-1})(X,Y,H)[(\dot{X},\dot{Y},\dot{H})] \\
&= (U_{\bot}\dot{X})H(V+V_{\bot}Y)^T + (U+U_{\bot}X)H(V_{\bot}\dot{Y})^T+
(U+U_{\bot}X)\dot{H}(V+V_{\bot}Y)^T.
\end{align*}
Assume that
\begin{align}\label{igualdad}
(U_{\bot}\dot{X})H(V+V_{\bot}Y)^T + (U+U_{\bot}X)H(V_{\bot}\dot{Y})^T+
(U+U_{\bot}X)\dot{H}(V+V_{\bot}Y)^T = 0.
\end{align}
Multiplying on the left by $U^+$ and on the right by $(V^+)^T$, we obtain
$\dot{H} = 0.$
Multiplying on the left by $U_{\bot}^{+}$ and on the right by $(V^+)^T$ we deduce
that $\dot{X}H= 0,$ that is, $\dot{X}=0.$
Finally, multiplying on the left by $U^{+}$ and on the right by $(V_{\bot}^+)^T$
we obtain $H\dot{Y}^T= 0,$ and hence $\dot{Y}= 0.$ Thus,
$D({i} \circ \theta_{Z}^{-1})(X,Y,H)$ is a linear isomorphism 
from $\mathbb{R}^{(n-r)\times r} \times \mathbb{R}^{(m-r)\times r} \times\mathbb{R}^{r \times r}$ to  
$D({i} \circ \theta_{Z}^{-1})(X,Y,H)[\mathbb{R}^{(n-r)\times r} \times \mathbb{R}^{(m-r)\times r} \times \mathbb{R}^{r \times r}]$
for each
$(X,Y,H) \in \mathbb{R}^{(n-r)\times r} \times \mathbb{R}^{(m-r)\times r} \times \mathrm{GL}_r.$
The inverse function theorem says us that 
$({i} \circ \theta_{Z}^{-1})$ is a diffeomorphism from $\mathbb{R}^{(n-r)\times r} \times \mathbb{R}^{(m-r)\times r} \times \mathrm{GL}_r$ to $\mathcal{U}_Z = ({i} \circ \theta_{Z}^{-1})(\mathbb{R}^{(n-r)\times r} \times \mathbb{R}^{(m-r)\times r} \times \mathrm{GL}_r)$ and in particular, it is a topological homeomorphism.
In consequence, the map $i$ is an embedding.
\end{proof}

The tangent space  to $\mathcal{M}_r(\mathbb{R}^{n\times m})$ at $Z = UGV^T$, which is the image through $T_Z i$ of the tangent space in local coordinates $\mathbb{T}_{Z}\mathcal{M}_r(\mathbb{R}^{n\times m})=\mathbb{R}^{(n-r)\times r} \times \mathbb{R}^{(m-r)\times r}\times \mathbb{R}^{r\times r}$, is 
\begin{align*}
{T}_Z \mathcal{M}_r(\mathbb{R}^{n\times m}) 
&= \{ U_{\bot} \dot X GV^T  +U G (V_\bot \dot Y)^T +  U \dot G V^T : \dot X \in \mathbb{R}^{(n-r)\times r}, \dot Y \in \mathbb{R}^{(m-r)\times r} ,\dot G \in  \mathbb{R}^{r\times r}  \},
\end{align*}
and can be decomposed into a vertical tangent space 
$$
{T}_Z^V \mathcal{M}_r(\mathbb{R}^{n\times m}) =  \{  U \dot G V^T : \dot G \in  \mathbb{R}^{r\times r}  \},
$$
and an horizontal tangent space
$$
{T}_Z^H \mathcal{M}_r(\mathbb{R}^{n\times m}) =  \{U_{\bot} \dot X GV^T  +U G (V_\bot \dot Y)^T   :\dot X \in \mathbb{R}^{(n-r)\times r},  \dot Y \in \mathbb{R}^{(m-r)\times r}   \}.
$$

\subsection{Lie group structure of neighbourhoods $\mathcal{U}_Z$}
We here prove that $\mathcal{M}_r(\mathbb{R}^{n\times m})$ has locally the structure of a Lie group by proving that the neighbourhoods $\mathcal{U}_Z$ can be identified with Lie groups. \\

Let $Z = UGV^T \in \mathcal{M}_r(\mathbb{R}^{n\times m})$. We first note that 
$\mathcal{U}_Z$ can be identified with 
$\mathcal{S}_U \times \mathcal{S}_V \times \mathrm{GL}_r$, with $\mathcal{S}_U$ and $\mathcal{S}_V$ defined by \eqref{SZLiegroup}. Noting that $\mathcal{S}_U$ and $\mathcal{S}_V$ can be identified with Lie groups $\mathcal{G}_U$ and $\mathcal{G}_V$ defined in \eqref{GZ}, we then have that $\mathcal{U}_Z$ can be identified with a product of three Lie groups, which is a Lie group with the  
group operation $\odot_Z$ given by
\begin{align*}
(\exp(U_{\bot}XU^+),\exp(V_{\bot}YV^+),G) &\odot_Z (\exp(U_{\bot} X'U^+),\exp(V_{\bot} Y'V^+), G') \\
&= (\exp(U_{\bot}(X+ X')U^+),\exp(V_{\bot}(Y+ Y')V^+),G G'). 
\end{align*}
It allows us to define a group operation $\star_Z$  over $\mathcal{U}_Z$
 defined for $W= \theta_Z^{-1}(X,Y,G)$ and $ W'= \theta_Z^{-1}( X', Y',G')$  by
\begin{align}
W \star_Z  W' 
& = \theta_Z^{-1}(X+ X',Y+ Y',G G'), \label{group operation UZ}
\end{align}
and to state the following result. 
\begin{theorem}\label{th:groupstructure2}
Let $Z = UGV^T \in \mathcal{M}_r(\mathbb{R}^{n\times m})$.
Then
the set $\mathcal{U}_Z$ together with the group operation $\star_Z$ defined by \eqref{group operation UZ}
 is a Lie group with identity element $UV^T$,  
 and the map $\eta_Z : \mathcal{U}_Z \to \mathcal{G}_U \times \mathcal{G}_V \times \mathrm{GL}_r$
given by  
$$
\eta_Z(\theta_Z^{-1}(X,Y,H)) = (\exp(U_\bot X U^+),\exp(V_\bot Y V^+),H)
$$
is a Lie group isomorphism. 
\end{theorem}


\end{document}